\documentclass[12pt]{amsart}
\usepackage[margin=2cm]{geometry}
\usepackage[utf8]{inputenc}
\usepackage[english]{babel}
\usepackage{quiver}
\usepackage{amsthm}
\usepackage{thmtools}
\usepackage{mathtools}
\usepackage{enumerate}
\usepackage{tikz-cd}
\usepackage[hidelinks]{hyperref}
\usepackage[nameinlink,capitalize,noabbrev]{cleveref}
\usepackage[style=alphabetic,backend=bibtex,maxalphanames=4,maxnames=10]{biblatex}
\makeatletter
\let\@afterindenttrue\@afterindentfalse
\@afterindentfalse
\makeatother

\numberwithin{equation}{section}
\newtheorem{theorem}[equation]{Theorem}

\newtheorem{lemma}[equation]{Lemma}
\theoremstyle{definition}
\newtheorem{definition}[equation]{Definition}
\newtheorem{example}[equation]{Example}
\newtheorem{construction}[equation]{Construction}
\newtheorem{remark}[equation]{Remark}
\usepackage{amsmath}
\usepackage{amssymb}
\usepackage{graphicx}
\usepackage{hyperref}

\newcommand{\N}{\mathbb{N}}
\newcommand{\inr}{\mathsf{inr}}
\newcommand{\id}{\mathsf{id}}
\newcommand{\inl}{\mathsf{inl}}

\newcommand{\colim}{\mathrm{colim}}

\newcommand{\glue}{\mathsf{glue}}

\newcommand{\pprod}{\widehat \times}
\newcommand{\smsh}{\wedge}

\DeclareMathOperator{\fw}{FW}
\DeclareMathOperator*{\bigast}{\raisebox{-0.6ex}{\scalebox{2.5}{$\ast$}}}

\title{Path spaces of pushouts}
\author{David Wärn}

\begin{document}
\begin{abstract}
Given a span of spaces, one can form the homotopy pushout and then take the
homotopy pullback of the resulting cospan. 
We give a concrete description of this pullback as the colimit of a sequence of approximations,
using what we call \emph{the zigzag construction}.
We also obtain a description of loop spaces of homotopy pushouts.
Using the zigzag construction, we reproduce generalisations of the Blakers--Massey
theorem and fundamental results from Bass--Serre theory.
We also describe the loop space of a wedge and show that it splits after suspension.
Our construction can be interpreted in a large class of $\infty$-categories
and in homotopy type theory, where it resolves the long-standing open problem
of showing that a pushout of 0-types is 1-truncated.
The zigzag construction is closely related to the James construction, but
works in greater generality.
\end{abstract}
\maketitle

\section{Introduction}

Given a (homotopy) pushout square of spaces
\[
\begin{tikzcd}	
A \arrow[r] \arrow[d] & C \arrow[d]\\
B \arrow[r] & D
\arrow[ul, phantom, "\ulcorner" , very near start]
\end{tikzcd}	
\]
it is often desirable to understand the (homotopy) pullback $B \times_D C$.
Considered together with its map to the product $B \times C$, this pullback
describes \emph{path spaces} of $D$. 
The situation is perhaps best understood
by considering a pushout of \emph{sets} viewed as discrete spaces. A span of
sets and functions, $B \leftarrow A \to C$, describes a bipartite multigraph, with vertex
sets $B$, $C$ and edge set $A$. The pushout is the geometric realisation of
this graph, and path spaces of the pushout are described by the \emph{free
groupoid} on the graph.
An object of this free groupoid is an element of the disjoint union
$B \sqcup C$.
Every morphism in the free groupoid can be
represented by a path in the graph, zigzagging between $B$ and $C$, and two
paths describe the same morphism if they are equal modulo
\emph{backtracking}, i.e.\ going back and forth along the same edge. The
pullback $B \times_C D$ is thus the set of all zigzags from $B$ to $C$
modulo backtracking.
For pushouts of general spaces, the situation is a priori more complicated;
one ought to consider a free $\infty$-groupoid, with \emph{spaces} of morphisms
rather than plain sets.

In the case of the loop space $\Omega \Sigma X$ of the suspension of a pointed 
connected space $X$, a relatively simple description is offered by the James
construction~\cite{james}. Here $\Omega \Sigma X$ is the free $\infty$-group
on $X$ (as a \emph{pointed} space), 
which when $X$ is connected coincides with the free $\infty$-monoid on
$X$. The free monoid on a pointed \emph{set} $X$ is given by finite sequences of elements
of $X$ modulo inserting the basepoint of $X$ in some position.
The James construction gives essentially the same description of $\Omega \Sigma X$
for general pointed connected spaces $X$.

The James construction was given a purely homotopical description by
Brunerie~\cite{james-brunerie}. In this description, the key insight is that
the free $\infty$-monoid $JX$ on $X$ is freely generated by a point
$\varepsilon : JX$ and a pointed action of $X$ on $JX$, expressing that we can append
an element of $X$ to an element of $JX$, to get an element of $JX$, 
and that appending the basepoint of $X$ does nothing. 
This can be seen as a variant of what is often called `van der Waerden's trick'~\cite{waerden}.
Remarkably, this avoids mentioning composition on $JX$, let alone
higher $\infty$-monoid laws. This gives a \emph{recursive} universal property of
$JX$, since the action takes an element of $JX$ as input.
A priori this recursive nature makes analysing $JX$ more difficult.
Fortunately $JX$ also admits a description as the colimit of a sequence of approximations,
with each approximation given by a pushout.
This can be compared with the description of the naturals $\N$ as
a nested union of finite sets $\{0\}$, $\{0,1\}$, $\{0,1,2\}$ etc.;
note that the naturals have a recursive universal property coming 
from the successor self-map $\N \to \N$.

Kraus and von Raumer gave a recursive universal property for general pullbacks of 
pushouts~\cite{kvr19}, similar to the universal property of $JX$.
In a sense this universal property encapsulates descent for pushouts.
Again its recursive nature a priori makes pullbacks of pushouts difficult to analyse.

In this paper we introduce the \emph{zigzag construction} to address this
issue. Given a pushout square as above, we build a sequence of spaces $ R_1
\to R_3 \to R_5 \to \cdots $ such that $R_1$ is $A$ and the colimit
of the sequence is the pullback $B \times_D C$. Morally,
$R_n$ is the space of zigzags from $B$ to $C$ of length at most $n$ modulo
backtracking. Note that any such zigzag has odd length. The map $R_{n-2}
\to R_n$ is a pushout of a map $J_n \to Z_n$ where $Z_n$ is the space of
all zigzags of length exactly $n$ and $J_n$ is, morally, the space of such
zigzags that contain a backtracking in some position. More precisely, 
each fibre of $J_n \to Z_n$
is an $(n-1)$-fold join of fibres of the diagonals of $A \to B$ 
and $A \to C$. These diagonals express the possibility of two edges with
a common endpoint coinciding, i.e.\ the possibility of a backtracking.
A zigzag of length $n$ has $n-1$ positions where it might backtrack;
the join roughly corresponds to the disjunction of these possibilities.
The zigzag construction similarly gives a description of pullbacks like
$B \times_D B$, $C \times_D C$, $B \times_D A$, and $C \times_D A$.
It also gives a description of the fibre of any of these pullbacks
over a point of the corresponding product, e.g.\ the fibre of $B \times_D B$
over a point of $B \times B$; these fibres are path spaces of $D$.

Even more generally, given a map of spans from $Q \leftarrow P \to R$ to
$B \leftarrow A \to C$, we have an induced map of pushouts
$t : Q \sqcup^P R \to B \sqcup^A C$, and the zigzag construction
describes the pullback of $B \leftarrow A \to C$ along $t$.
In this way we can for example describe the fibre of the `pinch' map
$X \sqcup^A 1 \to \Sigma A$ from a cofibre to a suspension; 
cf.\ the relative James construction defined by Gray~\cite{relative-james}.

The join of spaces is a homotopical analogue of the disjunction of
propositions: the join of propositions viewed as spaces is their disjunction.
This fact can be used to show that if both maps $A \to B$ and $A \to C$
are 0-truncated, then all the maps in the zigzag construction are
monomorphisms. (We explain the meanings of these notions in \cref{section:embeddings}.)
In particular, this means that the cartesian gap map $A \to B \times_C D$ is
a monomorphism and that the maps $B \to D$, $C \to D$ are 0-truncated.
These facts recover fundamental results from Bass--Serre theory,
including the normal form theorem describing the fundamental group of a graph.
In particular, we show that if both maps $A \to B$ and $A \to C$ are 0-truncated and $B$, $C$ are 
1-truncated, then $D$ is also 1-truncated. This resolves a long-standing open
problem in homotopy type theory.

One can use basic connectivity estimates including join connectivity to show
that if at least one of the maps $A \to B$ and $A \to C$ have connected fibres,
then the connectivity of $R_n \to R_{n+2}$ grows linearly with $n$. In
this sense we get increasingly accurate approximations to the pullback $B
\times_D C$. This recovers the classical Blakers--Massey theorem
and its generalisation to $\infty$-topoi and general modalities
due to Anel, Biedermann, Finster and Joyal~\cite{abfj}.

This paper is written with three audiences in mind: those interested in
ordinary spaces, those interested in more general $\infty$-categories such
as $\infty$-topoi, and
those interested in homotopy type theory (HoTT).
The construction originates in and uses many ideas from HoTT, but we have translated
the arguments to use more conventional higher categorical, or \emph{diagrammatic}, 
language.
While general higher categorical reasoning remains inaccessible in HoTT,
we use only an elementary fragment of higher category theory, 
staying inside one category throughout and avoiding mention of general diagrams.
The zigzag construction can be interpreted in any $\infty$-category
with pullbacks, pushouts that are universal and satisfy descent,
and universal sequential colimits.
This includes the $\infty$-category of spaces and any other $\infty$-topos.
\subsection*{Outline}
The rest of the paper is structured as follows. 
\begin{itemize}
\item In \cref{preliminaries}, we 
present the setting of the construction and discuss
some prerequisite results on pullbacks, pushouts, and sequential colimits.
\item In \cref{section:zigzag}, we present the zigzag construction and
develop some tools for analysing it.
\item In \cref{connectivity}, we make same basic connectivity estimates and
explain how to recover the Blakers--Massey theorem.
\item In \cref{section:embeddings}, we give an application of a rather different sort,
describing pushouts of spans of 0-truncated maps.
\item In \cref{taking-fibres}, we explain how to describe path spaces rather than pullbacks,
and describe some specific examples of loop spaces.
\item In \cref{loop-wedge}, we describe the loop space of a wedge and show that it splits
after suspension.
\item Finally in \cref{conclusion}, we sketch some directions for further research.
\end{itemize}
\subsection*{Acknowledgements}

The author would like to thank Christian Sattler, Hugo Moeneclaey, Ulrik
Buchholtz, and Gregory Arone for helpful discussions, and Yan Jiatong and Ali
Caglayan for reporting typos in earlier versions of this document. In
particular, Christian Sattler provided useful ideas for translating the
arguments from type-theoretic to diagrammatic language, and
Ulrik Buchholtz explained how to describe HNN extensions using the
zigzag construction.

\section{Preliminaries}\label{preliminaries}

In this section, we explain the setting we are working in and the ingredients for
our construction. This setting can be read as a collection of familiar facts 
about the $\infty$-category of spaces, or as a description of the kind of
categories in which the construction can be instantiated.
We stress that no knowledge of type theory or higher topos theory should be
necessary to understand our constructions.
We are also not committed to any particular model of $\infty$-categories.

For the rest of this paper, we work in a fixed $\infty$-category $\mathcal E$,
	where we take $\infty$-category to mean $(\infty,1)$-category.
We refer to objects of $\mathcal E$ simply as objects, and any diagram
we draw is a diagram in $\mathcal E$. 
We call a map in $\mathcal E$ with a two-sided inverse an \emph{equivalence}.
We now
spell out the assumptions made on $\mathcal E$. First, we assume
that pullbacks and pushouts exist, in the higher
categorical sense. We
denote the pullback of a cospan $X \to Z \leftarrow Y$ by $X \times_Z Y$, and
the pushout of a span $B \leftarrow A \to C$ by $B \sqcup^A C$. Note that both
of these construction depend on the maps involved. Our notation forgets these
maps, but they are usually clear from context. Similarly, when we draw a square- or
cube-shaped shaped diagram, we have in mind that there are specified witnesses of commutativity,
but we will not make explicit mention of these
homotopies. 
We call a square cartesian if it is a pullback square, 
cocartesian if it is a pushout square, and bicartesian if it is both.
We use the notation $0$ for the initial object of $\mathcal E$ if
it exists, and $1$ for the terminal object. We denote
the coproduct of $A$ and $B$ by $A \sqcup B$.
Given a map $A \to B$, we say that $A$ is an object \emph{over} $B$.
Similarly, given a natural transformation from a diagram $A$ to a diagram $B$
we say that $A$ is a diagram over $B$. 

The universal properties of pullbacks and pushouts mean that 
in any commutative square
\[
\begin{tikzcd}	
A \arrow[r] \arrow[d] & C \arrow[d]\\
B \arrow[r] & D
\end{tikzcd}	
\]
we have unique maps $A \to B \times_D C$ and $B \sqcup^A C \to D$ making
the appropriate diagrams commute. We refer to these maps as the
cartesian gap map and cocartesian cogap map of the square, or
gap map and cogap map for short.

\subsection*{Mather's cube theorems}
We assume pushouts are universal and satisfy descent.
Explicitly this means that Mather's cube theorems are satisfied:
given a commutative cube
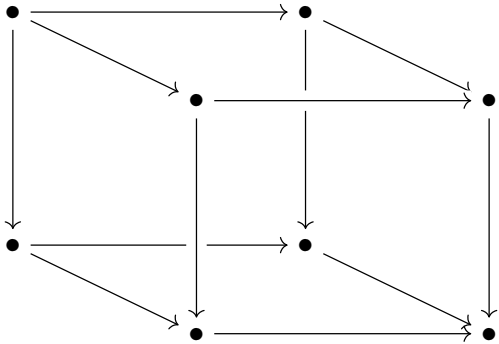
\[\begin{tikzcd}[ampersand replacement=\&,cramped]
	\bullet \&\&\& \bullet \\
	\&\& \bullet \&\&\& \bullet \\
	\\
	\bullet \&\&\& \bullet \\
	\&\& \bullet \&\&\& \bullet
	\arrow[from=5-3, to=5-6]
	\arrow[from=4-4, to=5-6]
	\arrow[from=4-1, to=5-3]
	\arrow[from=2-6, to=5-6]
	\arrow[from=1-4, to=4-4]
	\arrow[from=1-4, to=2-6]
	\arrow[from=1-1, to=4-1]
	\arrow[from=1-1, to=2-3]
	\arrow[from=1-1, to=1-4]
	\arrow[from=4-1, to=4-4]
	\arrow[from=2-3, to=5-3, crossing over]
	\arrow[from=2-3, to=2-6, crossing over]
\end{tikzcd}\]
in which the bottom face is cocartesian and 
the back and left faces are cartesian, the top face is cocartesian if
and only if the front and right faces are cartesian.
The backward direction says that pushouts are universal, i.e.\ stable under pullback.
The forward direction says that pushouts satisfy descent.

In homotopy type theory, descent corresponds to \emph{large elimination}.
This is the assertion that a type family over a pushout can be defined using the
universal property of the pushout, i.e.\ by pattern matching
on the constructors of the pushout.
This holds in particular in any setting with enough univalent universes,
since in this case a type family can be defined by mapping into the universe.
Universality of pushouts is usually understood in terms of the 
\emph{flattening lemma}~\cite{hottbook}.
It is a consequence of the existence of dependent products;
whenever dependent products exist, base change is a left adjoint and 
so preserves colimits.

\subsection*{Sequential colimits}
A \emph{sequence} consists
of an object $A_n$ for each $n : \N$ together with a map $A_n \to A_{n+1}$ for
each $n$. We denote the colimit of this sequence by $A_\infty$. 
We sometimes refer to the map $A_n \to A_\infty$ as the \emph{transfinite composition}
of $A_n \to A_{n+1} \to A_{n+2} \to \cdots$, and sometimes say that
the sequence is a \emph{filtration} of $A_\infty$.
We assume that sequential colimits are universal, i.e.\ stable under pullback. 
Explicitly, this means that if $(A_n)_{n : \N}$ is a sequence
and $X \to A_\infty$ is some object over the colimit, then the canonical map
$\colim_n (A_n \times_{A_\infty} X)\to X$ is an equivalence.
We also make use of the following strengthening.
\begin{lemma}\label{cofinality}
Let $\mathcal E$ be an $\infty$-category with pullbacks and universal sequential colimits.
Suppose given a diagram in $\mathcal E$ of the following shape,
where $(A_n)_{n : \N}$ and $(B_n)_{n : \N}$ are sequences with colimits $A_\infty$ and $B_\infty$.
\[\begin{tikzcd}[ampersand replacement=\&,cramped]
	{B_0} \& {B_1} \& \cdots \& {B_\infty} \& Y \\
	{A_0} \& {A_1} \& \cdots \& {A_\infty} \& X
	\arrow[from=1-1, to=2-1]
	\arrow[from=2-1, to=2-2]
	\arrow[from=1-1, to=1-2]
	\arrow[from=1-4, to=2-4]
	\arrow[from=2-4, to=2-5]
	\arrow[from=1-4, to=1-5]
	\arrow[from=1-5, to=2-5]
	\arrow[curve={height=6pt}, from=2-2, to=2-4]
	\arrow[curve={height=18pt}, from=2-1, to=2-4]
	\arrow[curve={height=-6pt}, from=1-2, to=1-4]
	\arrow[curve={height=-18pt}, from=1-1, to=1-4]
	\arrow[from=1-2, to=2-2]
	\arrow[from=1-2, to=1-3]
	\arrow[from=2-2, to=2-3]
\end{tikzcd}\]
If for infinitely many $n$, the gap map $B_n \to A_n \times_X Y$ is an equivalence,
   then so is the gap map $B_\infty \to A_\infty \times_X Y$.
\end{lemma}
\begin{proof}
Say $B_{n_i} \to A_{n_i} \times_X Y$ is an equivalence for each $i$ with
$n_0 < n_1 < \cdots$. Then $A_\infty$ and $B_\infty$ are the colimits of
$(A_{n_i})_{i : \N}$ and $(B_{n_i})_{i : \N}$ since the colimit
of a sequence agrees with the colimit along any subsequence.
By universality, we have
\[
	A_\infty \times_X Y 
		\simeq \colim_i (A_{n_i} \times_{A_\infty} (A_\infty \times_X Y))
		\simeq \colim_i (A_{n_i} \times_X Y)
		\simeq \colim_i B_{n_i}
		\simeq B_\infty.
\]
\end{proof}
\begin{remark}
The notion of a sequence becomes subtle in
an elementary $\infty$-topos and in homotopy type theory.
In a category with a natural numbers object $\N$,
we can define an \emph{internal} sequence to be an object $A$ over $\N$
together with an endomorphism $f : A \to A$ over the successor endomorphism $s : \N \to \N$.
A colimit of such an internal sequence is a coequaliser of 
$f$ and the identity map on $A$.
For contrast, let us call a sequence indexed by the external 
set of natural numbers an \emph{external} sequence.
Working with internal sequences is in a sense more permissive
than working with external sequences:
elementary $\infty$-topoi in the sense of Rasekh \cite{nima:elementary} do not in general admit
sequential colimits in the external sense, but do in the internal sense.
The discrepancy between internal and external natural numbers is the source of 
the problem of defining general higher structures in homotopy type theory 
\cite{higher-structures}.
In this paper, we work with external sequences, but the
constructions go through just as well with internal sequences, provided that
we have access to some mechanism for defining internal sequences by recursion
(e.g. a universe closed under appropriate operations).
In homotopy type theory, one has no choice but to work with internal sequences.
\end{remark}

\begin{definition}
Given two maps $X \to Z$ and $Y \to Z$ with common codomain, their
\emph{fibrewise join} $X \ast_Z Y \to Z$ is the cogap map of the
pullback square on $X \to Z \leftarrow Y$.
\end{definition}
This means that $X \ast_Z Y$ is the pushout $X \sqcup^{X \times_Z Y} Y$.
We call this a fibrewise join because the fibre over a point of $Z$
is a join of fibres of $X \to Z$ and $Y \to Z$; this follows from 
universality of pushouts. Also by universality of pushouts, the
fibrewise join is pullback stable in the sense that for any map $A \to Z$,
we have $(X \ast_Z Y) \times_Z A \simeq (X \times_Z A) \ast_A (Y \times_Z A)$.
The fibrewise join is closely related to the pushout product, which
features less prominently in our constructions.
\begin{definition}
Given two maps $f : A \to B$ and $g : X \to Y$, their
\emph{pushout product} $f \pprod g$
is the cogap map $(A \times Y) \sqcup^{A \times X} (B \times X) \to B \times Y$.
\end{definition}
\begin{definition}
Given pointed objects $A$ and $B$, the \emph{wedge inclusion} $A \vee B \to A \times B$
is the pushout product of $A \to A \times B$ and $B \to A \times B$.
Thus $A \vee B \simeq A\sqcup^1 B$.
\end{definition}
\begin{definition}
Given a map $f : A \to B$, the \emph{cofibre} of $f$ is the pushout $B \sqcup^A 1$.
\end{definition}
\begin{definition}
Given an object $A$, its \emph{suspension} $\Sigma A$ is the pushout $1 \sqcup^A 1$.
\end{definition}
\begin{definition}
Given objects $A$ and $B$, their \emph{smash product} $A \smsh B$
is the cofibre of the wedge inclusion $A \vee B \to A \times B$.
\end{definition}

A fundamental fact valid in any $\infty$-category is that (co)cartesian 
squares can be pasted. For cocartesian squares this means that in any 
diagram as follows, if the left square is cocartesian then the outer 
square is cocartesian if and only if the right square is cocartesian.
\[
\begin{tikzcd}
A \arrow[r] \arrow[d] & 
B \arrow[r] \arrow[d] & 
C \arrow[d] \\
X \arrow[r] & 
Y \arrow[r] 
\arrow[ul, phantom, "\ulcorner" , very near start]
& 
Z 
\end{tikzcd}
\]
We will use also the following more refined statement.

\begin{lemma}\label{cogap}
Consider a composite square in $\mathcal E$ as below.
\[
\begin{tikzcd}
A \arrow[r] \arrow[d] & 
B \arrow[r] \arrow[d] & 
C \arrow[d] \\
X \arrow[r] & 
Y \arrow[r] & 
Z 
\end{tikzcd}
\]
The cogap map $X \sqcup^A C \to Z$ of the outer square is the composite
of a pushout of the cogap map $X \sqcup^A B \to Y$ of the left square
followed by the cogap map $Y \sqcup^B C \to Z$ of the right square.
\end{lemma}
\begin{proof}
Consider the following diagram.
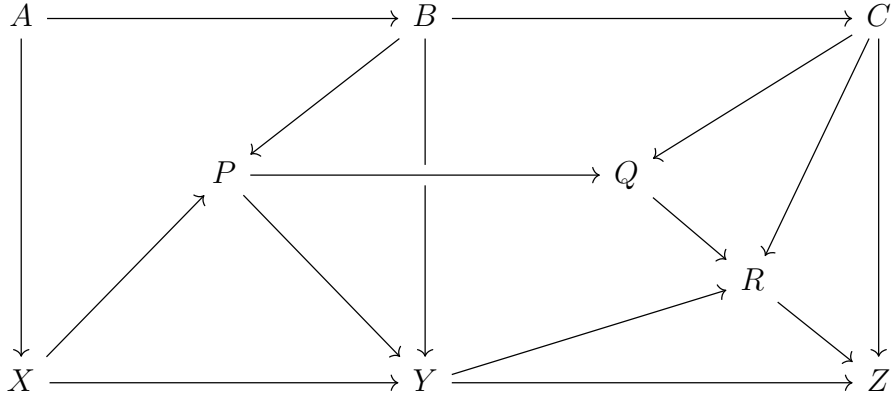
\[\begin{tikzcd}[ampersand replacement=\&,cramped]
	A \&\&\& B \&\&\& C \\
	\&\&\&\& Q \\
	\&\& P \&\&\& R \\
	X \&\&\& Y \&\&\& Z
	\arrow[from=1-1, to=1-4]
	\arrow[from=1-4, to=4-4]
	\arrow[from=4-1, to=4-4]
	\arrow[from=1-1, to=4-1]
	\arrow[from=4-4, to=4-7]
	\arrow[from=1-7, to=4-7]
	\arrow[from=3-6, to=4-7]
	\arrow[from=1-4, to=1-7]
	\arrow[from=2-5, to=3-6]
	\arrow[from=3-3, to=4-4]
	\arrow[from=4-1, to=3-3]
	\arrow[from=1-4, to=3-3]
	\arrow[from=3-3, to=2-5, crossing over]
	\arrow[from=4-4, to=3-6]
	\arrow[from=1-7, to=2-5]
	\arrow[from=1-7, to=3-6]
\end{tikzcd}\]
Here $P$ is $X \sqcup^A B$, $Q$ is $X\sqcup^A C$, and $R$ is $Y \sqcup^B C$.
The square $BCPQ$ is cocartesian by (reverse) pasting.
Thus $PQYR$ is also cocartesian by pasting.
This means that $Q \to R$ is a pushout of $P \to Y$. Since $Q \to Z$ is a composite
of $Q \to R$ and $R \to Z$, we are done.
\end{proof}


\section{The zigzag construction}\label{section:zigzag}

For the rest of this section, we consider a fixed pushout square as shown below.
\[
\begin{tikzcd}	
A \arrow[r] \arrow[d] & C \arrow[d]\\
B \arrow[r] & D
\arrow[ul, phantom, "\ulcorner" , very near start]
\end{tikzcd}	
\]
Our goal is to describe the pullback of this square along $B \to D$.
More generally, we consider the pullback along any map 
$S \to D$ where $S$ is the pushout of a span over $B \leftarrow A \to C$ 
and $S \to D$ is the induced map on pushouts. We can obtain the map $B \to D$ in this
way as $B$ is the pushout of $B \leftarrow 0 \to 0$, or of $B \leftarrow A \to A$.
\begin{construction}\label{making-cartesian}
Given a map of spans from $Q \leftarrow P \to R$ to $B \leftarrow A \to C$,
we can factorise it in the following two ways:
\begin{enumerate}[(i)]
\item via the span 
$Q^l \leftarrow P^l \to R^l$ where $Q^l \coloneqq Q$, $P^l \coloneqq Q \times_B A$,
and $R^l \coloneqq P^l \sqcup_P R$. 
We refer to this as \emph{making the left leg cartesian}.
\[
\begin{tikzcd}	
Q \arrow[d,"\sim"{sloped}] &
P 
\arrow[l] \arrow[r]
\arrow[d] &
R \arrow[d] \\
Q^l \arrow[d] &
P^l
\arrow[dl, phantom, "\llcorner" , very near start]
\arrow[l] \arrow[r]
\arrow[d] &
R^l \arrow[d] 
\arrow[ul, phantom, "\ulcorner" , very near start]
\\
B & A
\arrow[l] \arrow[r] &
C 
\end{tikzcd}	
\]
\item 
via the span 
$Q^r \leftarrow P^r \to R^r$ where 
$R^r \coloneqq R$, $P^r \coloneqq R \times_D A$,
and $Q^r \coloneqq P^r \sqcup_P Q$. 
We refer to this as making the \emph{right} leg cartesian.
\[
\begin{tikzcd}	
Q \arrow[d] &
P 
\arrow[l] \arrow[r]
\arrow[d] &
R \arrow[d,"\sim"{sloped}] \\
Q^r \arrow[d] 
\arrow[ur, phantom, "\urcorner" , very near start]
&
P^r
\arrow[dr, phantom, "\lrcorner" , very near start]
\arrow[l] \arrow[r]
\arrow[d] &
R^r \arrow[d] 
\\
B & A
\arrow[l] \arrow[r] &
C 
\end{tikzcd}	
\]
\end{enumerate}
\end{construction}
\begin{remark}\label{idempotent}
If $P \to Q$ is already cartesian over $A \to B$, then 
(i) above does nothing, i.e.\ the span $Q^l \leftarrow P^l \to R^l$ is
equivalent to $Q \leftarrow P \to R$. 
Similarly if $P \to R$ is cartesian over $A \to C$, then (ii)
does nothing. In particular, both (i) and (ii) are idempotent.
\end{remark}
\begin{lemma}\label{pushout-unchanged}
The maps of spans from $Q \leftarrow P \to R$ to
$Q^l \leftarrow P^l \to R^l$ and $Q^r \leftarrow P^r \to R^r$ both
induce equivalences on pushouts. That is, making a leg cartesian leaves pushouts unchanged.
\end{lemma}
\begin{proof}
By pushout pasting. 
We display the pasting diagram for making the left leg cartesian
below; the other case is symmetric. Note that $Q = Q^l$.
\[
\begin{tikzcd}	
P \arrow[r] \arrow[d]&
P^l
\arrow[r]
\arrow[d] &
Q \arrow[d] 
\\
R 
\arrow[r]
&
R^l 
\arrow[r]
\arrow[ul, phantom, "\ulcorner" , very near start]
&
R^l \sqcup_{P^l} Q
\arrow[ul, phantom, "\ulcorner" , very near start]
\end{tikzcd}\]
\end{proof}
\begin{construction}
[The zigzag construction]
\label{zigzag-construction}
Given a span $Q_0 \leftarrow P_0 \to R_0$ over $B \leftarrow A \to C$,
we define a sequence of spans over $B \leftarrow A \to C$
denoted $Q_n \leftarrow P_n \to R_n$ for $n : \N$,
		as follows.
The zeroth span is already given.
We define $Q_1 \leftarrow P_1 \to R_1$ by applying \cref{making-cartesian}~(i)
to the zeroth span. We define $Q_2 \leftarrow P_2 \to R_2$ by applying \cref{making-cartesian}~(ii)
to $Q_1 \leftarrow P_1 \to R_1$, and so on, making the left leg cartesian for odd $n$
and the right left cartesian for even $n \ge 2$.
\end{construction}
Below is a picture of the first few steps of the zigzag construction in the
general case and in the case where we start from $B \leftarrow 0 \to 0$.
Following this picture, we call the maps $P_n \to P_{n+1}$, $Q_n \to Q_{n+1}$,
and $R_n \to R_{n+1}$ the \emph{vertical} maps in the zigzag construction.
The fact that $Q_n \simeq Q_{n+1}$ for $n$ even and $R_n \simeq R_{n+1}$ for $n$ odd
means that we have two names for the same thing. 
We prefer to use the \emph{first} name, i.e.\ we index $Q$ by even integers and $R$ by odd integers.
For consistency we thus also prefer to refer to $R_0$ as $R_{-1}$ (this is anyway the initial object
in cases of interest).
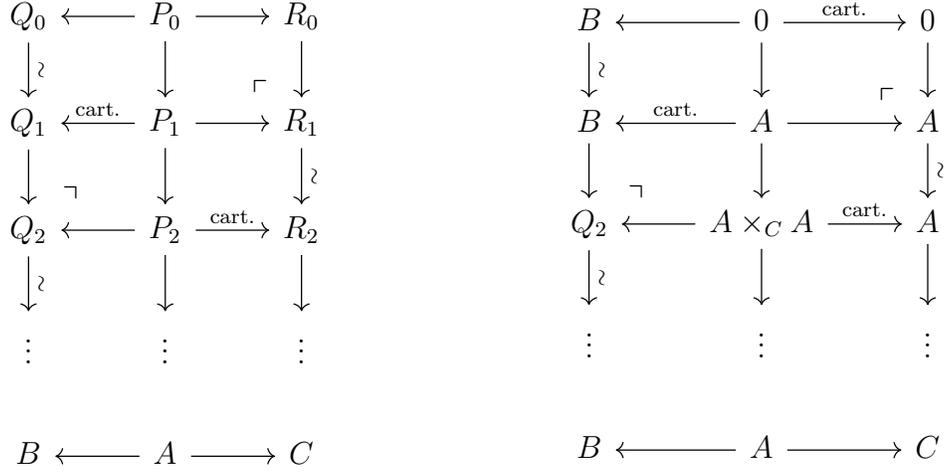
\[
\begin{tikzcd}	
Q_0 \arrow[d,"\sim"{sloped}] &
P_0 
\arrow[l] \arrow[r]
\arrow[d] &
R_0 \arrow[d] \\
Q_1 \arrow[d] &
P_1
\arrow[l,"\textup{cart.}", swap] \arrow[r]
\arrow[d] &
R_1 \arrow[d,"\sim"{sloped}] 
\arrow[ul, phantom, "\ulcorner" , very near start]
\\
Q_2 \arrow[d,"\sim"{sloped}] 
\arrow[ur, phantom, "\urcorner" , very near start]
&
P_2
\arrow[l] \arrow[r,"\textup{cart.}"]
\arrow[d] &
R_2 \arrow[d] 
\\
\vdots & \vdots & \vdots \\
B & A
\arrow[l] \arrow[r] &
C   & & &  
\end{tikzcd}	
\begin{tikzcd}	
B \arrow[d,"\sim"{sloped}] &
0
\arrow[l] \arrow[r, "\textup{cart.}"]
\arrow[d] &
0 \arrow[d] \\
B \arrow[d] &
A
\arrow[l,"\textup{cart.}", swap] \arrow[r]
\arrow[d] &
A \arrow[d,"\sim"{sloped}] 
\arrow[ul, phantom, "\ulcorner" , very near start]
\\
Q_2 \arrow[d,"\sim"{sloped}] 
\arrow[ur, phantom, "\urcorner" , very near start]
&
A \times_C A
\arrow[l] \arrow[r,"\textup{cart.}"]
\arrow[d] &
A \arrow[d] 
\\
\vdots & \vdots & \vdots \\
B & A
\arrow[l] \arrow[r] &
C 
\end{tikzcd}	
\]
\begin{remark}
There is an arbitrary choice in the construction, to make the left leg
cartesian instead of the right leg in the first step. The other choice is of course
also possible. In the examples we consider, one of the legs of the zeroth span
is already cartesian. 
In this case the choice only introduces an offset in the indexing of the sequence,
by \cref{idempotent}.
\end{remark}
\begin{theorem}\label{main-theorem}
Let $\mathcal E$ be an $\infty$-category with
pullbacks, pushouts, and sequential colimits.
Suppose that pushouts are universal and satisfy descent and that
sequential colimits are universal.
Suppose as in \cref{zigzag-construction} we are
given a span $Q_0 \leftarrow P_0 \to R_0$ over $B \leftarrow A \to C$ in $\mathcal E$,
and denote the pushouts $Q_0 \sqcup^{P_0} R_0$ and $B \sqcup^A C$ by $S$ and $D$ respectively.
Let $Q_\infty \leftarrow P_\infty \to R_\infty$
denote the sequential colimit of the sequence
$(Q_n \leftarrow P_n \to R_n)_{n \in \mathbb N}$ from \cref{zigzag-construction}.
Then we have the following cube
where the horizontal faces are pushouts and all four vertical faces are pullbacks.
In particular we have $Q_\infty \simeq S \times_D B$, $R_\infty \simeq S \times_D C$,
   and $P_\infty \simeq S \times_D A$.
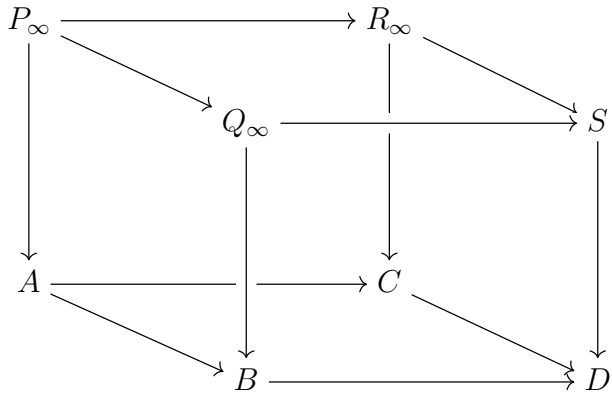
\[\begin{tikzcd}[ampersand replacement=\&,cramped]
	{P_\infty} \&\&\& {R_\infty} \\
	\&\& {Q_\infty} \&\&\& S \\
	\\
	A \&\&\& C \\
	\&\& B \&\&\& D
	\arrow[from=5-3, to=5-6]
	\arrow[from=4-1, to=5-3]
	\arrow[from=4-4, to=5-6]
	\arrow[from=1-4, to=2-6]
	\arrow[from=1-1, to=2-3]
	\arrow[from=1-1, to=4-1]
	\arrow[from=1-1, to=1-4]
	\arrow[from=2-6, to=5-6]
	\arrow[from=1-4, to=4-4]
	\arrow[from=4-1, to=4-4]
	\arrow[from=2-3, to=5-3, crossing over]
	\arrow[from=2-3, to=2-6, crossing over]
\end{tikzcd}\]
\end{theorem}
\begin{proof}
We first claim that $S$ is the pushout of the colimit span $Q_\infty \leftarrow P_\infty \to R_\infty$.
By applying interchange of pushouts and sequential colimits to the sequence of
spans $Q_n \leftarrow P_n \to R_n$, we see that the pushout of the colimit span
is the sequential colimit of the pushout sequence. By \cref{pushout-unchanged}, this
is a sequence of equivalences, starting with $S$. Since equivalences are closed under transfinite composition, the colimit is $S$, as an object over $D$.
This means that we at least have a cube as shown, 
where the top and bottom faces are pushout squares.
We claim that the left face is cartesian.
This follows from \cref{cofinality} since $P_n \to Q_n$ is cartesian over $A \to B$ for
all odd $n$. 
Similarly the back face is cartesian since $P_n \to R_n$ is cartesian over $A \to C$
for all even $n \ge 2$.
Finally descent tells us that the front and right faces are cartesian.
\end{proof}

It will be useful to have a more precise understanding of what changes from one step
to the next in \cref{zigzag-construction}.
This is controlled by the maps $P_{n-1} \to P_n$:
among the other two maps, $Q_{n-1} \to Q_n$ and $R_{n-1} \to R_n$,
one is an equivalence and the other one is a pushout of $P_{n-1} \to P_n$.
The following somewhat technical lemma provides a way to recursively describe $P_{n-1} \to P_n$ as a pushout.

\begin{lemma}\label{pushout-step}
\begin{enumerate}[(i)]
\item Let $Q \leftarrow P \to R$ be a span over $B \leftarrow A \to C$ with
$P \to R$ cartesian over $A \to C$.
Suppose given a pushout square as follows.
\[
\begin{tikzcd}	
X \arrow[r] \arrow[d] & P \arrow[d]\\
Y \arrow[r] & P^l
\arrow[ul, phantom, "\ulcorner" , very near start]
\end{tikzcd}	
\]
In particular we have a composite map $Y \to P^l \to A$, and also a map $Y \to C$.
Then we have a pushout square
\[
\begin{tikzcd}	
X' \arrow[r] \arrow[d] & P^l \arrow[d]\\
Y \times_C A \arrow[r] & P^{lr}
\arrow[ul, phantom, "\ulcorner" , very near start]
\end{tikzcd}	
\]
where $X'$ is the fibrewise join of $X \times_C A \to Y \times_C A$ and $Y \to Y \times_C A$.
Moreover, the composite $Y \times_C A \to P^{lr} \to A$ is projection onto the second factor.

\item Dually, if $Q \leftarrow P \to R$ is a span over $B \leftarrow A \to C$
with $P \to Q$ cartesian over $A \to B$ and we have a pushout square
\[
\begin{tikzcd}	
X \arrow[r] \arrow[d] & P \arrow[d]\\
Y \arrow[r] & P^r
\arrow[ul, phantom, "\ulcorner" , very near start]
\end{tikzcd}	
\]
then we have another pushout square
\[
\begin{tikzcd}	
X' \arrow[r] \arrow[d] & P^r \arrow[d]\\
Y \times_B A \arrow[r] & P^{rl}
\arrow[ul, phantom, "\ulcorner" , very near start]
\end{tikzcd}	
\]
where $X'$ is the fibrewise join of $X \times_B A \to Y \times_B A$ and
$Y \to Y \times_B A$, and again the composite $Y \times_B A \to P^{rl} \to A$
is projection onto the second factor.
\end{enumerate}
\end{lemma}
\begin{proof}
We prove only the first statement since the second statement is symmetric.
By construction we have
$P^{lr} \simeq R^l \times_C A$ and $R^l \simeq P^l \sqcup^P R$.
Universality of pushouts gives the right cocartesian square in the following diagram.
\[ \begin{tikzcd}
P
\arrow[r]
\arrow[d]
\arrow[dr, phantom, "\lrcorner" , very near start]
& P \times_C A
\arrow[r]
\arrow[d]
& R \times_C A
\arrow[d]
\\ P^l
\arrow[r]
& P^l \times_C A
\arrow[r]
& P^{lr}
\arrow[ul, phantom, "\ulcorner" , very near start]
\end{tikzcd} \]
The left square is a pullback by pasting, and the top composite is an equivalence.
Consider the result of applying \cref{cogap} to the above diagram.
The cogap map of the outer square is simply the bottom composite
$P^l \to P^{lr}$.
The cogap map of the left square is an equivalence, so
we find that $P^l \to P^{lr}$ is a pushout of the cogap map of the left square.
Thus it suffices to show that the cogap map of the left square is a pushout of
$X' \to Y \times_C A$.
Consider now the following diagram.
\[\begin{tikzcd}[ampersand replacement=\&,cramped]
	X \&\&\& Y \\
	\&\& {X \times_C A} \&\&\& {X'} \&\&\& {Y \times_C A} \\
	\\
	P \&\&\& {P^l} \\
	\&\& {P \times_C A} \&\&\& \bullet \&\&\& {P^l \times_C A}
	\arrow[from=2-6, to=2-9]
	\arrow[from=5-6, to=5-9]
	\arrow[from=2-6, to=5-6]
	\arrow[from=5-3, to=5-6]
	\arrow[from=4-1, to=5-3]
	\arrow[from=1-1, to=2-3]
	\arrow[from=1-4, to=2-6]
	\arrow[from=1-1, to=1-4]
	\arrow[from=4-4, to=5-6]
	\arrow[from=1-4, to=4-4]
	\arrow[from=1-1, to=4-1]
	\arrow[from=2-9, to=5-9]
	\arrow[from=4-1, to=4-4]
	\arrow[from=2-3, to=5-3, crossing over]
	\arrow[from=2-3, to=2-6, crossing over]
\end{tikzcd}\]
We have $X \simeq (X \times_C A) \times_{Y \times_C A} Y$ by pullback pasting,
so the top square of the cube is simply the pushout square defining $X'$.
The bottom square is also cocartesian.
The back square is the cocartesian square given by assumption.
The whole cube is induced by functoriality of pushouts.
By pushout pasting, the front face of the cube is cocartesian.
By universality of pushouts, the outer front square is cocartesian.
By reverse pushout pasting, the rightmost square is cocartesian,
as needed.
\end{proof}
If we drop the assumption that the right leg $P \to R$ is cartesian in \cref{pushout-step}~(i),
we can still say something useful if a bit more complicated. 
Namely, $P^l \to P^{lr}$ is a composite of the pushout of the gap
map $P \to R \times_C A$ along $P \to P^l$ with a pushout of the map
$X' \to Y \times_B A$ defined as in (i). We will not make use of this.

\begin{construction}\label{zigzag-analysis}
Importantly, \cref{pushout-step} can be iterated.
Consider \cref{zigzag-construction} starting from a span $Q_0 \leftarrow P_0 \to R_0$ over
$B \leftarrow A \to C$ with $P_0 \to R_0$ cartesian over $A \to C$.
We define for every $n \ge 1$ two objects $J_n$, $Z_n$ with a map $J_n \to Z_n$ and a pushout square of the following form.
\[
\begin{tikzcd}	
J_n \arrow[r] \arrow[d] & P_{n-1} \arrow[d]\\
Z_n \arrow[r] & P_n
\arrow[ul, phantom, "\ulcorner" , very near start]
\end{tikzcd}	
\]
For $n = 1$ we take the trivial pushout square witnessing that $P_0 \to P_1$ is a pushout of itself.
For $n \ge 2$ even we apply \cref{pushout-step} (i) to the pushout square for $n-1$, and
for $n \ge 3$ odd we similarly apply \cref{pushout-step} (ii).
\end{construction}

Thus the whole zigzag construction is controlled by the maps $J_n \to Z_n$:
the map $P_n \to P_\infty$ is a transfinite composition of pushouts of the maps
$J_m \to Z_m$ for $m \ge n+1$.
Similarly, $Q_n \to Q_\infty$ for $n$ even is transfinite composition of pushouts of the maps $J_m \to Z_m$ for $m \ge n+2$ even,
and $R_n \to R_\infty$ for $n$ odd is a transfinite composition of pushouts of the maps $J_m \to Z_m$ for $m \ge n+2$ odd.
It remains to give a more explicit description of $J_n \to Z_n$.

By construction, $Z_n$ is simply an iterated pullback $Q_0 \times_B A \times_C A \cdots$
containing $n$ copies of $A$, so that for example $Z_1$ is $Q_0 \times_B A \simeq P_1$.
Thinking of an element of $A$ as an edge with one endpoint in $B$ and the other in $C$,
we can picture an element of $Z_n$ as a zigzag of $n$ edges, starting in $B$,
where the first vertex (in $B$) is furthermore labelled by an element of the fibre of $Q_0$.
More formally, $Z_n$ is the limit of the following diagram (in the case $n = 3$; in general the diagram contains
$n$ copies of $A$).
\[\begin{tikzcd}[ampersand replacement=\&,cramped]
	{Q_0} \&\& A \&\& A \&\& A \\
	\& B \&\& C \&\& B
	\arrow[from=1-1, to=2-2]
	\arrow[from=1-3, to=2-2]
	\arrow[from=1-3, to=2-4]
	\arrow[from=1-5, to=2-4]
	\arrow[from=1-5, to=2-6]
	\arrow[from=1-7, to=2-6]
\end{tikzcd}\]

Universality of pushouts means that $J_n$ is a \emph{fibrewise join} of $n$ maps into $Z_n$.
The first of these $n$ maps is the pullback of $P_0 \to P_1$ along the projection $Z_n \to P_1$;
if $P_0$ is empty we can simply omit this map from the fibrewise join.
For $n\ge 2$, the next map is a pullback of the diagonal $A \to A \times_C A$ along
the projection $Z_n \to A \times_C A$; this corresponds to looking at the first
two edges in a zigzag and asking if they coincide as edges with a common endpoint in $C$.
All other maps are defined similarly, as pullbacks of $A \to A \times_C A$
or of $A \to A \times_B A$, expressing the possibility of a zigzag having a backtracking in
one of $n-1$ possible positions.

The existence of the map $J_n \to P_{n-1}$ expresses \emph{confluence of reduction}:
if a zigzag has a backtracking in different positions, then one can reduce in any
of those positions to get the same result modulo further reduction.

The zigzag construction is a priori asymmetric, in the sense that
a priori one gets different filtrations of $B \times_D C$ depending
on whether we start from $B \leftarrow 0 \to 0$ and consider $S \times_D C$,
or start from $0 \leftarrow 0 \to C$ and consider $S \times_D B$.
We expect, but do not prove, that these two filtrations are actually the same.
\cref{zigzag-analysis} goes some way toward proving this.

\begin{remark}
Let us make a comparison with the James construction.
Let $A$ be a pointed object of $\mathcal E$ and let $J_nA$ be the sequence given by the James construction
as in \cite{james-brunerie}.
Thus $J_0 A \to J_1A$ is the basepoint inclusion $1 \to A$ and 
for $n\ge 1$, $J_n A \to J_{n+1} A$ is a certain pushout of 
the pushout product of $J_{n-1}A \to J_n$ and $1 \to A$.
Because the pushout product is functorial in pushout squares~\cite[Lemma 2.2.3]{acyclic},
whenever $J_{n-1}A \to J_n A$ is a pushout of some map $X \to Y$,
we have that $J_n A \to J_{n+1} A$ is a pushout of the pushout product of
$X \to Y$ and $1 \to A$.
This fact is analogous to \cref{pushout-step}.
It can also be iterated: $J_{n-1} A \to J_n A$ is a pushout of
the pushout product of $1 \to A$ with itself $n$ times.
This iterated pushout product is a homotopical description a \emph{fat wedge} inclusion
$\fw^n(A) \to A^n$.
Salient properties of the James construction can be deduced from the fact that
$J_{n-1}A \to J_n A$ is a pushout of $\fw^n(A) \to A^n$.

We can thus understand the James construction as 
adding at stage $n$ all lists of $n$ elements of
$A$, and gluing those lists where some element is trivial onto the previous step.
Algebraically this corresponds to a reduction using a \emph{unit} law of a monoid, 
like $a_1 \cdot 1 \cdot a_3 = a_1 \cdot a_3$.
In contrast, the zigzag construction adds at stage $n$ zigzags of exactly $n$ edges,
and glues those zigzags where some adjacent pair of edges cancel onto the previous step.
Algebraically this corresponds to a reduction using an \emph{inverse} law of a groupoid, 
like $a_1 a_2^{-1} a_2 a_4^{-1} = a_1 a_4^{-1}$.
The inverses appear because zigzags go back and forth between $B$ and $C$;
interpreted as paths, every other edge must be inverted.
\end{remark}

\begin{remark}
One can give a more type-theoretic description of the zigzag construction.
A span $B \leftarrow A \to C$ can equivalently be described as a pair of types $B$ and $C$
together with a type family $A(b,c)$ over $b : B$, $c : C$.
If $Q$ and $R$ are objects over $B$ and $C$ then it is natural to also represent these
as type families over $B$ and $C$.
Completing this picture to a span with the right leg cartesian means
that for all $b : B$, $c : C$, $a : A(b,c)$, we have a map $R(c) \to Q(b)$.
In this way we do not have to mention $P$ as a type on its own.
Dually, completing the picture to a span with the left leg cartesian means
that for $b : B$, $c : C$, $a : A(b,c)$ we have a map $Q(b) \to R(c)$.

The zigzag construction starting from a span with the right leg cartesian can thus
be understood as building a sequence of type families $Q_n$ over $B$, indexed by even $n \ge 0$
and a sequence of type families $R_n$ over $C$, indexed by odd $n \ge -1$, so 
that for $b : B$, $c : C$, $a : A(b, c)$ we have
for $n$ even a map $Q_n(b) \to R_{n+1}(c)$ and for $n$ odd a map
$R_n(c) \to Q_{n+1}(b)$.
These maps fit into the following commutative diagram, again given $a  : A(b,c)$.
\[
\begin{tikzcd}
& Q_0(b) \arrow[dr] \arrow[rr] & & Q_2(b) \arrow[dr] \arrow[rr] & & \cdots\\
R_{-1}(c) \arrow[ur] \arrow[rr] & & R_1(c) \arrow[ur] \arrow[rr] & & R_3(c)  \arrow[ur] \arrow[rr] & & \cdots
\end{tikzcd}
\]
In the limit, we get a span over $B \leftarrow A \to C$ with both legs cartesian, which
corresponds to having for $a : A(b,c)$ an \emph{equivalence} $Q(b) \simeq R(c)$.
In our proof of \cref{main-theorem} we argued that
given a sequence of spans over $B \leftarrow A \to C$ where every other span has left leg cartesian
and every other span has right leg cartesian, in the colimit both legs are cartesian.
This corresponds to the fact that the diagram above induces an equivalence between
the colimit of the top and bottom sequences (indeed, both colimits are equivalent to the colimit of
the interleaved sequence $R_{-1}(c) \to Q_0(b) \to R_1(c) \to \cdots$).

We can also describe the meaning of the equivalence $Q_\infty(b) \simeq R_\infty(c)$:
descent ensures that $Q_\infty$ and $R_\infty$ extend to define a type family $S$ over the pushout $B \sqcup^A C$,
so that $Q_\infty(b) \simeq S(\inl(b))$ and $R_\infty(c) \simeq S(\inr(c))$,
and the equivalence $Q_\infty(b) \simeq R_\infty(c)$ corresponds to the \emph{transport} of this
type family $S$ along a path constructor $\glue(a) : \inl(b) = \inr(c)$.
In particular, if $S(x)$ is a type family like $\inl(b_0) = x$ for a fixed $b_0 : B$,
	then transport in $S$ corresponds to path composition.
\end{remark}

\section{Connectivity estimates}\label{connectivity}

The zigzag construction produces a sequence of approximations to pullbacks,
e.g.\ the sequence $P_0 \to P_1 \to \cdots$ provides a sequence of
approximations to $S \times_D A$. It is natural to ask how quickly this
sequence `converges', i.e.\ how well does $P_n$ approximate $S \times_D A$.
We obtain results in this direction from~\cref{zigzag-analysis}.
More formally, we can ask for a lower bound on the connectivity of the transfinite
composition $P_n \to S \times_D A$. Thus suppose our category $\mathcal E$
comes with a notion of $n$-connected map for $n \ge -2$.
This is the case e.g.\ if $\mathcal E$ is an elementary $\infty$-topos.
We follow the indexing convention of~\cite{abfj}, so that 
every map is $(-2)$-connected, and $(-1)$-connected maps are (essentially) surjective maps.
For us, only the following properties of $n$-connected maps are relevant:
\begin{enumerate}[(i)]
\item For each $n$, $n$-connected maps are closed under pushout.
\item For each $n$, $n$-connected maps are closed under pullback.
\item For each $n$, $n$-connected maps are closed under transfinite composition.
\item For each $n$, if $f : A \to B$ is $(n+1)$-connected then $\Delta f: A \to A \times_B A$ is
$n$-connected.
\item Join connectivity holds: if $X \to Z$ is $m$-connected and $Y \to Z$ is $n$-connected,
	then their fibrewise join $X \ast_Z Y \to Z$ is $(m+n+2)$-connected.
\end{enumerate}

Suppose that $\mathcal E$ has a notion of $n$-connected map satisfying the properties above,
as well as our usual assumption that $\mathcal E$ has pullbacks, 
pushouts that are universal and satisfy descent, and universal sequential colimits.
Let $Q_0 \leftarrow P_0 \to R_0$ be a span over $B \leftarrow A \to C$
with the right leg cartesian.
Write $S$ for $Q_0 \sqcup^{P_0} R_0$ and $D$ for $B \sqcup^A C$.
Suppose that the gap map $P_0 \to Q_0 \times_B A$ is $m$-connected,
that the diagonal $A \to A \times_B A$ is $k$-connected, and
that the diagonal $A \to A \times_C A$ is $l$-connected.
This happens in particular if $A \to B$ is $(k+1)$-connected and $A \to C$ is $(l+1)$-connected.
Then join connectivity means that $J_{2n} \to Z_{2n}$ is $(m + n(l+2) + (n-1)(k+2))$-connected
and $J_{2n+1} \to Z_{2n+1}$ is $(m + n(k+l+4))$-connected.
Thus we have similar connectivity estimates for $P_{n-1} \to P_n$,
$Q_{2n-2} \to Q_{2n}$ and $R_{2n-1} \to R_{2n+1}$,
and for transfinite compositions $P_n \to S \times_D A$,
$Q_n \to S \times_D B$, and $R_n \to S \times_D C$.
Note in particular that the connectivity of $J_n \to Z_n$ grows linearly with $n$,
provided that at least one of $k$ and $l$ is greater than $-2$.

Of particular interest is the case where $Q_0 \leftarrow P_0 \to R_0$ is $B
\leftarrow 0 \to 0$. In this case we take $m = -2$ and have that $J_3 \to Z_3$
is $(k+l+2)$-connected.
Thus $R_1 \to R_3$ is also $(k+l+2)$-connected, and the later maps $R_3 \to R_5$ and
so on are even more connected. In particular, the transfinite composition
$R_1 \to R_\infty$ is $(k+l+2)$-connected.
This is precisely the gap map $A \to B \times_D C$, so our connectivity estimates recover
the Blakers--Massey theorem.
In the same way one recovers the more general statement for modalities
proved by Anel et al.~\cite{abfj}.
In fact the map $J_3 \to Z_3$ is precisely the \emph{relative pushout product}
that appears in their Theorem 4.1.1.
For each $n > 3$, we have that $J_n \to Z_n$ is a fibrewise join where one of the maps is a pullback of $J_3 \to Z_3$,
so it in particular belongs to the left class of any modality containing $J_3 \to Z_3$.

We offer the following attempt at an informal explanation of the Blakers--Massey theorem.
The gap map $A \to B \times_D C$ can be thought of as including the space
of `edges' into the space of all possible zigzags modulo backtracking from $B$ to $C$.
The Blakers--Massey theorem describes the extent to which it is true that
every zigzag modulo backtracking is given by a unique single edge.
The assumption that $A \to A \times_B A$ is connected corresponds to assuming that
any two edges with a common endpoint in $B$ are, to some degree of approximation, the same.
Any zigzag from $B$ to $C$ has odd length, so either it is exactly one or it is at least $3$.
In the latter case, the zigzag has at least two pairs of adjacent edges, where one pair
shares an endpoint in $B$ and the other pair shares an endpoint in $C$.
Our assumption on the connectivity of diagonals means that each one of these pairs is, 
to some degree of approximation, a backtracking.
Join connectivity means that to an even higher
degree of approximation, our zigzag has a backtracking \emph{somewhere}, meaning that the zigzag is equivalent to
a shorter one. Iterating this reasoning, every zigzag from $B$ to $C$ is to a certain degree of approximation
equivalent to one of length one.

The assumption on connectivity of $A \to A \times_B A$ or $A \to A \times_C A$
is necessary to get any meaningful connectivity estimates for the zigzag
construction. Indeed \cref{circle} shows that it can happen that no vertical map
in the zigzag construction is even $(-1)$-connected.

\section{Monomorphisms and 0-truncated maps}\label{section:embeddings}

An important fact about pushouts is that a (homotopy) pushout of sets is 1-truncated, i.e.\ its 
loop spaces are sets.
This can be understood as a \emph{coherence theorem} for $\infty$-groupoids: it says that
the free $\infty$-groupoid on a graph happens to be a $1$-groupoid.
This is analogous to the coherence theorem for bicategories.
Indeed, classically any pushout of sets is a disjoint union of wedges of spheres.
In this section, we introduce and prove a generalisation of this coherence theorem.
This statement captures in an elegant way fundamental results from Bass--Serre theory.
This means that in particular we reconstruct non-trivial results from combinatorial group theory.
All our arguments are based on the zigzag construction and have no other combinatorial input.

\begin{definition}
We say a map $f : X \to Y$ is $(-2)$-truncated if it is an equivalence. For $n \ge -1$ we say $f$
is $n$-truncated if its diagonal $\Delta f : X \to X \times_Y X$ is $(n-1)$-truncated.
We say an object $X$ is $n$-truncated if $X \to 1$ is.
\end{definition}
We call a $(-1)$-truncated map, i.e.\ a map whose diagonal is an equivalence, a \emph{monomorphism}.%
\footnote{In homotopy type theory, monomorphisms are usually referred to as \emph{embeddings}.}
One can think of $0$-truncated maps as covering maps, up to homotopy, or as faithful
functors.
Classically, every monomorphism in spaces is complemented, 
i.e.\ a coproduct inclusion $X \to X \sqcup Y$,
and a map is $0$-truncated if it is injective
on each $\pi_1$ and an isomorphism on $\pi_2$ and above.
We sometimes call $0$-truncated objects \emph{sets}, following the terminology of homotopy type theory.

In any elementary $\infty$-topos, $n$-connected maps and $n$-truncated maps
form the left and right classes of a factorisation system. In this sense
truncated maps are dual to connected maps. We have that $n$-truncated maps are
closed under pullback but generally not under pushout. 
For example, $1 \to S^2$ is a pushout of $S^1 \to 1$, 
and the latter is $1$-truncated but the former is not $n$-truncated for
any $n$. In light of this, it is perhaps surprising that monomorphisms \emph{are} closed under pushout.

\begin{lemma}\label{pushout-embedding}
Suppose $\mathcal E$ is an $\infty$-category with pullbacks, 
and pushouts that are universal and satisfy descent.
Suppose given a pushout square
\[
\begin{tikzcd}	
A \arrow[r, hook, "f"] \arrow[d] & C \arrow[d]\\
B \arrow[r, "g"] & D
\arrow[ul, phantom, "\ulcorner" , very near start]
\end{tikzcd}	
\]
in which the map $f$ is an monomorphism.
Then:
\begin{enumerate}[(i)]
\item $g$ is an monomorphism.
\item The square is cartesian.
\item The pullback $C \times_D C$ fits into the following pushout square.
\[
\begin{tikzcd}	
A \arrow[r, hook, "f"] \arrow[d] & C \arrow[d]\\
A \times_B A \arrow[r] & C \times_D C
\arrow[ul, phantom, "\ulcorner" , very near start]
\end{tikzcd}	
\]
\end{enumerate}
\end{lemma}

Informally, part (iii) expresses that two elements of $C$ are identified in $D$ if and only if either
they are already identified in $C$, or they are both in the image of $f$ and their preimages are
identified in $B$.

The first two statements are proved directly using descent
in~\cite{abfj}, and all of them are proved 
in~\cite[Proposition A.10]{partial-univalence}. 
We present a proof based on the zigzag construction.
\begin{proof}[Proof of \cref{pushout-embedding}]
Consider the zigzag construction starting from $B \leftarrow 0 \to 0$.
In this case $J_2 \to Z_2$ is precisely the diagonal 
$\Delta f : A \to A \times_C A$.
Since $f$ is a monomorphism, $J_2 \to Z_2$ is an equivalence.
Since equivalences are closed under pushout,
the whole map of spans from $Q_1 \leftarrow P_1 \to R_1$ to
$Q_2 \leftarrow P_2 \to R_2$ is an equivalence.
Thus the span $Q_1 \leftarrow P_1 \to R_1$ already has both
legs cartesian over $B \leftarrow A \to C$.
As in the proof of \cref{main-theorem}, this means that
$Q_1 \leftarrow P_1 \to R_1$ is equivalent to
$B \times_D B \leftarrow B \times_D A \to B \times_D C$.
In particular we have $B \simeq B \times_D B$ and
$A \simeq B \times_D C$.

For part (iii), we instead start from $A \leftarrow A \to C$.
We have $J_2 \simeq Z_2$ in this case as well so that $R_1 \simeq R_\infty$.
This gives the desired third pushout square.
\end{proof}


Note that if in \cref{pushout-embedding} $f$ is a \emph{complemented} monomorphism, i.e.\ $A \to A
\sqcup X$, then so is $g$. Namely, it is $B \to B \sqcup X$. In any
elementary $\infty$-topos, monomorphisms and complemented monomorphisms enjoy another good closure property
normally satisfied by the \emph{left} class of a factorisation system: they
are closed under fibrewise join~\cite{join}.
Assuming sequential colimits commute with pullback, we also have that
monomorphisms are closed under transfinite composition~\cite{sequential}.%
\footnote{To see this, note that given a sequence
	$A_0 \to A_1 \to A_2 \to \cdots$, the diagonal
	of its transfinite composition $A_0 \to A_\infty$
	is the transfinite composition of
	$A_0 \to A_0 \times_{A_1} A_0 \to A_0 \times_{A_2} A_0 \to \cdots$.
If all the maps in the original sequence are monomorphisms, then
all the maps in the latter sequence are equivalences.}
This gives the following result.
\begin{theorem}\label{0-truncated-pushout}
Let $\mathcal E$ be an $\infty$-category with pullbacks,
pushouts that are universal and satisfy descent, and sequential colimits that
commute with pullbacks.
Suppose moreover that monomorphisms in $\mathcal E$ are closed under fibrewise join.
Let 
\[
\begin{tikzcd}	
A \arrow[r, "g"] \arrow[d,"f"] & C \arrow[d, "k"]\\
B \arrow[r, "h"] & D
\arrow[ul, phantom, "\ulcorner" , very near start]
\end{tikzcd}	
\]
be a pushout square where $f$ and $g$ are both 0-truncated.
Then $h$ and $k$ are also $0$-truncated, and the gap map $A \hookrightarrow B \times_D C$ is
a monomorphism.
\end{theorem}
\begin{proof}
Consider the zigzag construction starting from $B \leftarrow 0 \to 0$.
Since monomorphisms are closed under pullback and fibrewise join, we have that
$J_n \to Z_n$ is a monomorphism for all $n \ge 1$, and so all the vertical maps in the zigzag construction
are monomorphisms. In particular $B \to B \times_D B$ is an monomorphism, i.e.\ $h$ is $0$-truncated,
	and $A \to B \times_D C$ is a monomorphism. The claim that $k$ is $0$-truncated is symmetric
	(i.e.\ proved by starting the zigzag construction from $0 \leftarrow 0 \to C$).
\end{proof}
Note that the same proof shows that \cref{0-truncated-pushout} also holds if we replace
`monomorphism' with `complemented monomorphism' and `$0$-truncated map' with
`map whose diagonal is a complemented monomorphism'.

A prototypical example of \cref{0-truncated-pushout} is as follows.
Let $G$, $H$, $N$ be groups and suppose we have injective group homomorphisms $N \to G$ and $N \to H$.
Then the deloopings $BN \to BH$ and $BN \to BG$ are $0$-truncated.
\cref{0-truncated-pushout} says that the pushout in spaces is 1-truncated; in fact it
is the delooping of the amalgamated product, $B(G \ast_N H)$. The fact that
the maps into the pushout are $0$-truncated means that $G$ and $H$ are subgroups of
$G \ast_N H$.
The fact that the gap map is a monomorphism means that $N$ is the intersection of $G$ and $H$
in $G \ast_N H$ (since taking loop spaces preserves pullback squares and
turns monomorphisms into equivalences).

More generally, any \emph{graph of groups} determines a span of 0-truncated maps.
Concretely, suppose $V$ and $E$ are sets, with maps $s, t : E \to V$,
and suppose we have for each $v \in V$ a group $G_v$ and for each $e : E$ a group
$H_e$ with injective group homomorphisms $H_e \to G_{s(e)}$ and $H_e \to G_{t(e)}$.
Let $B$ be the disjoint union $\bigsqcup_{v \in V} BG_v$ of deloopings, and
let $C$ be the disjoint union $\bigsqcup_{e \in E} BH_e$. Let $A = C \sqcup C$
with $A \to C$ the codiagonal and $A \to B$ given by $s$ on the first component and $t$ on the second.
Note that $A \to C$ is 0-truncated since all fibres are of the form $1 + 1$, and
$A \to B$ is 0-truncated since the group homomorphisms
$H_e \to G_{s(e)}$, $H_e \to G_{t(e)}$ are injective.
The pushout $B \sqcup^A C$ is the realisation of the fundamental groupoid of the graph of groups.
Equivalently, this is a coequaliser of $C \rightrightarrows B$, i.e.\ $\bigsqcup_{e \in E} BH_e \rightrightarrows \bigsqcup_{v \in V} BG_v$.
We give a more explicit description of the pushout in \cref{groupoid}.

\subsection*{The free higher group on a set in homotopy type theory}

The zigzag construction is in part motivated by the following question from
HoTT: for a set $A$, is the coequaliser $A \rightrightarrows 1$ a 
1-truncated type? In the case where $A$ has decidable equality,
Brunerie formalised a proof already in
2012~\cite{dec-freegroup}. This later appeared as Exercise 8.2
of~\cite{hottbook}, with the general case mentioned as an open question.
Kraus and Altenkirch~\cite{ka18} expressed the problem as describing the free
higher group on $A$. They also showed that the 1-truncation of the free higher group
is 0-truncated, using ideas about confluent rewriting, similar to the constructive proof
in~\cite[Chapter X]{bible} that a set embeds in its free group.
Christian Sattler proposed to understand the free higher group on a set by observing
that any set is a filtered colimit of finite sets, and using that suspension and
loop spaces commutes with filtered colimits.
The issue with this argument is that we have no known way to express colimits
of general diagrams in HoTT.

Another way to understand the problem is as follows. 
One might hope to describe the free higher group on a set by showing that its Cayley graph is a tree.
Given a graph $s, t : E \rightrightarrows V$ with $E$ and $V$ sets, 
we can classically say that $(V, E)$ is a tree if $V$ has at least one element
and any two elements of $V$ are joined by a unique undirected path without backtracking.
Trees also admit a homotopical description: a tree is a graph whose coequaliser is contractible.
Going from the first characterisation to the second is possible also in a constructive setting,
i.e.\ in any elementary $\infty$-topos or in homotopy type theory.
But the reverse is wildly false, essentially
because we cannot find enough paths without backtracking.
So we cannot in general hope to show that a coequaliser
is contractible by showing unique existence of reduced paths.
Instead, a reasonable constructive characterisation of trees is that they are graphs whose coequaliser
is connected and has trivial $\pi_1$, 
meaning that any cycle of positive length has a backtracking. 
Still, it is constructively a non-trivial fact that a tree in this sense has
contractible coequaliser.
It follows from the following theorem.
\begin{theorem}\label{pushout-truncated}
Working in homotopy type theory, consider a pushout square as follows.
\[
\begin{tikzcd}	
A \arrow[r, "g"] \arrow[d,"f"] & C \arrow[d, "k"]\\
B \arrow[r, "h"] & D
\arrow[ul, phantom, "\ulcorner" , very near start]
\end{tikzcd}	
\]
Suppose $f$ and $g$ are $0$-truncated and 
that $B$ and $C$ are both $n$-truncated with $n \ge 1$. Then $D$ is also $n$-truncated.
\end{theorem}
\begin{proof}
Given an element $d : D$, we have to show that the iterated loop space
$\Omega^{n+1}(D, d)$ is contractible.
Being contractible is a property, and $x$ is \emph{merely} of the form $h(b)$ or $k(c)$,
so we may without loss of generality suppose it is of the form $h(b)$.
Now since $h$ is 0-truncated by \cref{0-truncated-pushout}, it is an equivalence on second loop spaces,
	$\Omega^2(h) : \Omega^2(B, b) \simeq \Omega^2(D, h(b))$.
In particular we have $\Omega^{n+1}(B, b) \simeq \Omega^{n+1}(D, h(b))$.
The former is contractible since $B$ is $n$-truncated, so we are done.
\end{proof}
Note that any map of sets is $0$-truncated, so \cref{pushout-truncated} in particular 
shows that any pushout or coequaliser or suspension of
0-truncated types is 1-truncated.

\section{Taking fibres}\label{taking-fibres}
Say again we are given a pushout square as follows.
\[
\begin{tikzcd}	
A \arrow[r,"g"] \arrow[d,"f"] & C \arrow[d,"k"]\\
B \arrow[r,"h"] & D
\arrow[ul, phantom, "\ulcorner" , very near start]
\end{tikzcd}	
\]
The pullback $B \times_D C$ describes, as an object over $B \times C$, a family of \emph{path spaces}
of $D$. More precisely, suppose we have points $b : 1 \to B$ and $c : 1 \to C$.
Following the notation of homotopy type theory, write
$h(b) = k(c)$ for the pullback of $h \circ b$ and $k \circ c$ -- the space of paths
from $h(b)$ to $k(c)$.
Then $h(b) = k(c)$ is the fibre of $B \times_D C$ over $(b, c) : 1 \to B \times C$.
The zigzag construction starting from
$B \leftarrow 0 \to 0$
describes $B \times_D C$ as the sequential colimit of $R_1 \to R_3 \to \cdots$, and
since sequential colimits commute with pullbacks,
$h(b) = k(c)$ is the sequential colimit of the sequence of fibres
$R_1(b,c) \to R_3(b,c) \to \cdots$.
Here $R_1(b,c)$ is the fibre $A(b,c)$ of $A$ over $(b,c)$.
We can also pull back the pushout square from \cref{zigzag-analysis} along $(b,c)$ to
obtain the following pushout square, for $n$ odd.
\[
\begin{tikzcd}	
J_n(b,c) \arrow[r] \arrow[d] & R_{n-2}(b,c) \arrow[d]\\
Z_n(b,c) \arrow[r] & R_n(b,c)
\arrow[ul, phantom, "\ulcorner" , very near start]
\end{tikzcd}	
\]
Remarkably, this gives a description of the path space $h(b) = k(c)$ that is \emph{independent}
of the description of other path spaces; we can study each path space in isolation.
The fibre $Z_n(b,c)$ of $Z_n$ over $(b, c)$ is an iterated pullback
$b \times_B A \times_C A \cdots A \times_C c$ -- the space of length $n$ zigzags that start
in $b$ and end in $c$.

A similar analysis applies if we take the fibre of $B \times_D B$ over a point of $B \times B$,
or the fibre of $B \times_D A$ over a point of $B \times A$.
In the latter case, given $b : 1 \to B$ and $a : 1 \to A$, consider
the fibre $Z_n(b, a)$ of $Z_n$ over $(b,a)$.
It is an iterated pullback $b \times_B A \times_C A \cdots A \times_A a$,
with $n$ copies of $A$. This is the space of zigzags of length exactly $n$ that start in $b$
and end in \emph{the edge} $a$.
This is equivalent to the iterated pullback like $b \times_B A \cdots A \times_B f(a)$ for $n$ odd,
or $b \times_B A \cdots A \times_C g(a)$ for $n$ even -- zigzags of length exactly $n-1$
whose endpoint agrees with that of $a$.
Under this latter description of $Z_n(b, a)$, one should be careful to note that
$J_n(b, a)$ is the space of zigzags of length $n-1$ starting in $b$ and ending in an endpoint of $a$,
where \emph{either} the zigzag has a backtracking in some position, \emph{or} the last edge is $a$
(equivalently, the zigzag obtained by appending $a$ has a backtracking).

\begin{example}
Suppose that $B$ and $C$ are the terminal object $1$, so that $D$ is the suspension
$\Sigma A$.
In the zigzag construction starting from $1 \leftarrow 0 \to 0$ over
$1 \leftarrow A \to C$,
we have that 
 $Z_n$ is $A^n$ and that $J_n$ is the join of $n-1$ maps $A^{n-1} \to A^n$, where each map
duplicates one coordinate, e.g.\ $(a_1, a_2, a_3) \mapsto (a_1, a_2, a_2, a_3)$.
Here an element like $(a_1, a_2, a_3, a_4)$ of $A^n$ should be thought of as representing a path
like $a_1 a_2^{-1} a_3 a_4^{-1}$.
The zigzag construction gives a description of the loop space $\Omega \Sigma A$,
without assuming that $A$ is connected or has a basepoint:
it is the colimit of $Q_0 \to Q_2 \to Q_4 \to \cdots$, where $Q_{n-2} \to Q_n$ is a pushout of $J_n \to Z_n$.
If we do assume that $A$ has a basepoint $a : 1 \to A$, we get a slightly more refined description,
by taking the colimit of $P_1(a) \to P_2(a) \to P_3(a) \to \cdots$.
Here $P_1(a)$ is $1$ and $P_{n-1}(a) \to P_n(a)$ is a pushout of $J_n(a) \to Z_n(a)$, 
where $Z_n(a) = A^{n-1}$ and $J_n(a)$ is the fibrewise join of $n$ maps $A^{n-2} \to A^{n-1}$:
the $n-1$ maps that duplicate a coordinate, and the map which inserts $a$ in the last coordinate.
\end{example}

\begin{example}\label{circle}
Consider the suspension of a two-element set $A \coloneqq \{a, b\}$.
This suspension is a circle, so the loop space should be the integers. More precisely,
the suspension is equivalent to a circle provided that we choose a basepoint of $A$.
Consider the zigzag construction starting from $1 \leftarrow 0 \to 0$
over $1 \leftarrow A \to 1$.
Now $Z_n$ is simply the product $A^n$.
$J_n$ is the subset of $A^n$ consisting of lists where
some two adjacent elements are equal. The monomorphism $J_n \to A^n$ 
is complemented, and its complement consists of two elements
$(a, b, a,\cdots)$ and $(b, a, b, \cdots)$.
Thus all the vertical maps in the zigzag construction
simply add two new points.
The loop space $Q_\infty$ thus consists of words like
$1$, $ab^{-1}$, $ba^{-1}$, $ab^{-1}ab^{-1}$, $ba^{-1}ba^{-1}$, etc. 
These are precisely integer powers of $ab^{-1}$.
The path space $R_\infty$ consists of words like
$a$, $b$, $ab^{-1}a$, $b a^{-1} b$, etc.
The pullback $J_n(a) \to Z_n(a)$ of $J_n \to Z_n$ along $\{a\} \hookrightarrow \{a,b\}$
is also complemented, and its complement is just one point.
Thus in the sequence $P_1(a) \to P_2(a) \to P_3(a) \to \cdots$, each map adds just one new
point.
\end{example}

\begin{example}\label{groupoid}
Let us give a more precise description of what happens in \cref{0-truncated-pushout}
in the case where $A$, $B$, $C$ are groupoids.
In this case $f : A \to B$ and $g : A \to C$ are \emph{faithful} functors.
Given objects $b$ and $c$ of $B$ and $C$, we give a description of
$Z_3(b, c)$, i.e.\ $b \times_B A \times_C A \times_B A \times_C c$; 
one can obtain a similar description of $Z_n(b, c)$ for general $n$.
It turns out that $Z_3(b,c)$ is a set.
Explicitly, an element of $Z_3(b, c)$ can be represented by three objects $a_1$, $a_2$, $a_3$ of 
$A$ together with isomorphisms $p_1 : b \cong f(a_1)$, 
$p_2 : g(a_1) \cong g(a_2)$, $p_3 : f(a_2) \cong f(a_3)$, $p_4 : g(a_3) \cong c$.
However this representation is not exactly unique:
e.g. if $q : a_1 \cong a_1'$, then replacing $a_1$ with $a_1'$,
$p_1$ with $f(q) \circ p_1$ and $p_2$ with $p_2 \circ g(q)$
results in the same element of $Z_3$. 
It is the faithfulness of $f$ and $g$ that ensures that this set quotient 
is actually a homotopy quotient, so that $Z_3(b,c)$ is a set
(or more precisely, a thin groupoid).
Describing $J_3(b,c)$ is easier: it is the subset of $Z_3(b,c)$ consisting of those
elements where at least one of $p_2$, $p_3$, and $p_4$ lies in the image of $f$ or $g$
(again this is a property since $f$ and $g$ are faithful).

The homset $h(b) \cong k(c)$ of the pushout groupoid $D$ is simply the union of
a sequence of subsets $R_1(b,c) \hookrightarrow R_3(b,c) \hookrightarrow \cdots$.
This sequence is described quite explicitly by bicartesian squares, for $n \ge 3$ odd:
\[
\begin{tikzcd}	
J_n(b,c) \arrow[r] \arrow[d, hook] 
\arrow[dr, phantom, "\lrcorner" , very near start]
& R_{n-2}(b,c) \arrow[d, hook]\\
Z_n(b,c) \arrow[r] & R_n(b,c)
\arrow[ul, phantom, "\ulcorner" , very near start]
\end{tikzcd}	
\]
If $J_n(b,c) \hookrightarrow Z_n(b,c)$ is complemented (classically, a redundant assumption),
then this says that $R_n(b,c)$ is the union of $R_{n-2}(b,c)$ with the complement
$Z_n(b,c) \setminus J_n(b,c)$ -- the set of \emph{reduced} zigzags.
This can be compared with Higgins' description of the fundamental groupoid of a graph of 
groups~\cite{higgins}.
If we do not assume $J_n \to Z_n$ is complemented, we at least obtain the 
following characterisation of equality in $h(b) \cong k(c)$, using
\cref{pushout-embedding}:
two elements of $Z_n(b,c)$ define the same element of $h(b) \cong k(c)$ if and only if
either they are equal in $Z_n(b, c)$, or they both reduce
and the two reductions in $Z_{n-2}(b,c)$ define the same element of $h(b) \cong k(c)$.
\end{example}
\begin{example}
\cref{groupoid} gives a basepoint-free and constructive description of the fundamental groupoid
of any graph of groups.
For concreteness, we explain the special case of an HNN extension.
In this case, we start from a parallel pair of injective group homomorphisms
$\alpha, \beta : H \hookrightarrow G$.
The HNN extension $G \ast_H$ is a group with an injective group homomorphism $G \to G \ast_H$
and an element $t \in G \ast_H$ such that $t\alpha(h)t^{-1} = \beta(h)$ for $h \in H$.
It is defined by the following pushout square, 
where we can view $BG$ as the one-object groupoid with automorphism group $G$.
\[
\begin{tikzcd}	
BH \sqcup BG \arrow[r,"{\lfloor \beta , \id \rfloor}"] \arrow[d,"{\lfloor \alpha , \id \rfloor}"] 
& BG \arrow[d]\\
BG \arrow[r] & B(G\ast_H)
\arrow[ul, phantom, "\ulcorner" , very near start]
\end{tikzcd}	
\]
We can consider $BH \sqcup BG$ to be pointed by the basepoint of $BG$ and consider
the fibres of $P_1 \to P_2 \to \cdots$ over the corresponding point $(b,a)$ of $BG \times (BH \sqcup BG)$.
We then have the following description of $Z_n(b,a)$.
An element is represented by a word of length $n$, where the first letter is an element of $G$,
and every other letter is either $1$ or of the form $t^{\pm 1} g$ where $g \in G$ and the 
exponent of $t$ depends on the parity of the position of the letter.
Say the exponent is $+1$ if the parity of the position is even.
Two such words represent the same element of $Z_n(b,a)$ if they are related by moves like
$g t \alpha(h) g' = g \beta(h) t g'$ and
$g t^{-1} \beta(h) g' = g \alpha(h) t^{-1} g'$.
Such a word \emph{reduces}, i.e.\ defines an element of $J_n(b, a)$ 
if either it contains to successive letters $1$,
or it ends in a $1$, or it contains a pattern
$t \alpha(h) t^{-1}$ or a pattern $t^{-1} \beta(h) t$.
\end{example}

Note that while the above description of HNN extensions is obtained
mechanically from the zigzag construction, it is also in some sense more
complicated than the general zigzag construction. This complexity comes from
the explicit unfolding of what a homotopy pullback is, and from the specific
span whose pushout is $B(G \ast_H)$. This is all hidden in the general picture.

One could say that the zigzag construction is a refinement of the (groupoidal)
Seifert--van Kampen theorem. Whereas the Seifert--van Kampen theorem describes
the fundamental groupoid of a pushout, the zigzag construction roughly speaking
describes the fundamental $\infty$-groupoid. But even aside from this,
\cref{groupoid} shows that the zigzag construction gives a useful description
of the fundamental groupoid, at least in some cases.

\section{The loop space of a wedge}\label{loop-wedge}

The zigzag construction admits a particularly nice description
in the case of a wedge. The Hilton--Milnor theorem
describes the loop space of a wedge of suspensions of pointed connected spaces~
\cite{hilton-milnor,lavenir}.
Using the zigzag construction we can say something interesting also for a 
general wedge. In particular we will show that it splits after suspension.

\begin{definition}
Given an object $I$ and an $I$-indexed family $(A(i))_{i \in I}$ of pointed objects, the wedge
$\bigvee_{i \in I} A(i)$ is given by the following pushout.
\[
\begin{tikzcd}	
I \arrow[r] \arrow[d] & \bigsqcup_{i \in I} A(i) \arrow[d]\\
1 \arrow[r] & \bigvee_{i \in I} A(i)
\arrow[ul, phantom, "\ulcorner" , very near start]
\end{tikzcd}	
\]
\end{definition}

\begin{remark}
We should read $I$ as an object of $\mathcal E$ in the above definition.
This means that $A$ is defined by an object over $I$ 
(namely $\bigsqcup_{i \in I} A(i)$) with a section. 
For simplicity we the same notation one would normally use
for actual, external sets $I$.
For example, $\bigsqcup_{i \in I} \Omega A(i)$ denotes
the loop space object of $A$ \emph{in} the slice $\mathcal E/I$.
This is in line with homotopy type theory, where one would write
$\Sigma_{i : I} A(i)$ or $(i : I) \times A(i)$ in place of $\bigsqcup_{i \in I} A(i)$.
This notation makes implicit some stability under taking pullbacks:
if $f : J \to I$, then $\bigsqcup_{j \in J} \Omega A(f(j))$ can be defined
either by pulling back $A$ along $f$ to get a pointed object over $J$ and
then taking loops in $\mathcal E/J$, or as first taking loops in $\mathcal E/I$
and then pulling back.
If $I$ is an external set and $\mathcal E$ has
universal $I$-indexed coproducts, then there is also a corresponding internal $0$-truncated
object of $\mathcal E$: the coproduct of $I$-many copies of $1$.
\end{remark}

Now suppose $I$ is an object and $A$ is an $I$-indexed family of pointed objects.
Consider the zigzag construction applied to the pushout square defining the wedge
$\bigvee_{i \in I} A(i)$, starting from the span $1 \leftarrow 0 \to 0$
over $1 \leftarrow I \to \bigsqcup_{i \in I} A(i)$.
We are particularly interested in the sequence $Q_0 \to Q_2 \to \cdots$
whose colimit is the loop space $\Omega \bigvee_{i \in I} A(i)$.
It is worth keeping in mind that this loop space should be the free product,
or coproduct in the category of $\infty$-groups, of the loop spaces
$\Omega A(i)$, although we do not treat this idea formally.

Consider the map $J_{2n} \to Z_{2n}$.
We have that $Z_{2n}$ is an iterated pullback
$I \times_{\bigsqcup_{i \in I} A(i)} I \times_1 \cdots$.
The pullback $I \times_{\bigsqcup_{i \in I} A(i)} I$ is 
the relative loop space
$\bigsqcup_{i \in I} \Omega A(i)$.
Pullback over $1$ is simply product, so $Z_{2n}$ is an $n$-fold product
$(\bigsqcup_{i \in I} \Omega A(i))^n$, or equivalently
$\bigsqcup_{(i_1\ldots i_n)\in I^n} \Omega A(i_1) \times \cdots \times \Omega A(i_n)$.

Let us now describe $J_{2n}$. It is a join of $2n-1$ maps. Of these, $n$
are pullbacks of $I \to \bigsqcup_{i \in I} \Omega A(i)$, a kind of relative basepoint
inclusion, and
the other $n-1$ are pullbacks of the diagonal $I \to I \times I$.
The first kind of map expresses the possibility that given a word
$w = a_1 a_2 a_3 \cdots a_j \cdots $ with $a_j \in \Omega A_{i_j}$, 
it could be that $w$ reduces because some letter $a_j$ is trivial.
The second kind expresses the possibility that $i_j = i_{j+1}$, so
that $w$ reduces because $a_j$ and $a_{j+1}$ can be multiplied in $\Omega A_{i_j}$.

Now suppose that $I$ is 0-truncated and that the diagonal $I \to I \times I$ is complemented,
a classically redundant assumption. In this case, assuming coproducts are universal, $I^n$ decomposes as
a coproduct $E_n \sqcup U_n$, where $E_n$ consists of lists where
some adjacent elements are equal and $U_n$ is the complement.
This decomposition of $I^n$ induces a decomposition
$J_{2n} \simeq (J_{2n} \times_{I^n} E_n) \sqcup (J_{2n} \times_{I^n} U_n)$
by universality of coproducts.
We claim that one component in this decomposition is trivial: 
$J_{2n} \times_{I^n} E_n \to Z_{2n} \times_{I^n} E_n$ is an equivalence.
This is because, by definition of $E_n$, over $E_n$, 
the fibrewise join of the $n-1$ pullbacks of $I \to I \times I$ is an equivalence.
On the other hand, over $U_n$, $J_{2n}$ these $n-1$ pullbacks of $I \to I \times I$
are all empty, and $J_{2n}$ simply looks like a relative fat wedge.

In this way we see that $J_{2n} \to Z_{2n}$ and hence $Q_{2n-2} \to Q_{2n}$ is a pushout
of the following map, a relative fat wedge inclusion.
\begin{equation} \label{disj-fw}
\bigsqcup_{(i_1\ldots i_n) \in U_n} \fw(\Omega A_{i_1}\ldots \Omega A_{i_n})
	\rightarrow \bigsqcup_{(i_1\ldots i_n) \in U_n} \Omega A_{i_1} \times \cdots \times \Omega A_{i_n}
	\end{equation}
Here, for $X_1 \cdots X_n$ pointed objects, $\fw(X_1\ldots X_n)$ denotes the
fat wedge, the pushout product of basepoint inclusions
$1 \to X_i$. The map in \eqref{disj-fw} is really a relative fat wedge, computed
in the slice $\mathcal E/U_n$.
An important fact about fat wedge inclusions is that they split after suspension.
We say a map $f : X \to Y$ of pointed objects \emph{splits} if 
it is a pushout of
a basepoint inclusion $1 \to Z$
(along the basepoint $1 \to X$); in this case $Z$ is necessarily the cofibre of $f$.
Our goal for the rest of this section is to show that the suspension of the map \eqref{disj-fw}
splits. To this end, we first have to prove some results that are more or 
less well-known in a classical setting.
See \cite{polyhedral} for general splitting results for polyhedral products.

We use the following lemma without proof. It is proved diagrammatically in~\cite{james-splitting}.

\begin{lemma}\label{suspension-product}
Let $\mathcal E$ be an $\infty$-category with finite products and pushouts.
If $A$ and $B$ are pointed objects, then we have
$\Sigma (A \times B) \simeq \Sigma A \vee \Sigma B \vee \Sigma (A \smsh B)$
naturally in $A$ and $B$.
\end{lemma}

\begin{lemma}\label{times-splits}
Let $\mathcal E$ be an $\infty$-category with finite products and pushouts.
Let $f : A \to B$ be a map of pointed objects that splits after suspension.
Then so does $f \times X : A \times X \to B \times X$
for any pointed object $X$, naturally in $X$.
\end{lemma}
\begin{proof}
Let $C$ denote the cofibre of $f$ so that $\Sigma B \simeq \Sigma A \vee \Sigma C$.
We refer to the following diagram.
\[\begin{tikzcd}[ampersand replacement=\&,cramped]
	\& 1 \& {\Sigma X} \\
	1 \& {\Sigma A \vee \Sigma(A \smsh X)} \& {\Sigma(A \times X)} \\
	{\Sigma C \vee \Sigma(C \smsh X)} \& {\Sigma B \vee \Sigma(B \smsh X)} \& {\Sigma(B \times X)}
	\arrow[from=1-2, to=2-2]
	\arrow[from=1-2, to=1-3]
	\arrow[from=1-3, to=2-3]
	\arrow[from=2-2, to=2-3]
	\arrow[from=2-2, to=3-2]
	\arrow[from=3-2, to=3-3]
	\arrow[from=2-3, to=3-3]
	\arrow[from=3-1, to=3-2]
	\arrow[from=2-1, to=3-1]
	\arrow[from=2-1, to=2-2]
\end{tikzcd}\]
The right outer and top squares are pushouts by \cref{suspension-product}.
By reverse pushout pasting so is the bottom right square.
The bottom left square is a pushout by distributing $\vee$ over $\Sigma$ and $\smsh$
and using $\Sigma(B \smsh X) \simeq \Sigma B \smsh X$.
The bottom outer square gives the desired result.
\end{proof}
\begin{lemma}\label{pushout-product-splits}
Let $\mathcal E$ be an $\infty$-category with finite products and pushouts.
Let $f : A \to B$ and $g : X \to Y$ be maps of pointed objects.
Suppose that $\Sigma f$ and $\Sigma g$ split. Then so does
$\Sigma(f \widehat \times g)$.
\end{lemma}
Note that any map from $1$ splits, so in particular any
fat wedge inclusion $\fw(X_1 \ldots X_n) \to X_1 \times \cdots \times X_n$ splits after suspension.
\begin{proof}
Let $C$ be the cofibre of $f$ and $Z$ the cofibre of $g$.
We use \cref{times-splits} twice
and refer to the following diagram in which each square is cocartesian.
\[\begin{tikzcd}[ampersand replacement=\&,cramped]
	\& {} \\
	\& 1 \& {\Sigma(A \times X)} \& {\Sigma(B \times X)} \\
	1 \& {\Sigma Z \vee \Sigma (A \smsh Z)} \& {\Sigma(A \times Y)} \& \bullet \\
	{\Sigma(C \smsh Z)} \& {\Sigma Z \vee \Sigma(B \smsh Z)} \&\& {\Sigma(B \times Y)}
	\arrow["{\Sigma(f \widehat \times g)}", from=3-4, to=4-4]
	\arrow[from=3-3, to=3-4]
	\arrow[from=2-4, to=3-4]
	\arrow[from=2-3, to=3-3]
	\arrow[from=2-3, to=2-4]
	\arrow[from=2-2, to=3-2]
	\arrow[from=2-2, to=2-3]
	\arrow[from=3-2, to=3-3]
	\arrow[from=4-2, to=4-4]
	\arrow[from=3-2, to=4-2]
	\arrow[from=4-1, to=4-2]
	\arrow[from=3-1, to=4-1]
	\arrow[from=3-1, to=3-2]
\end{tikzcd}\]
The leftmost bottom square is obtained by distributing $\smsh$ over $\vee$.
The outer bottom square gives the desired result.
\end{proof}
\begin{lemma}\label{suspension-disj}
Let $\mathcal E$ be an $\infty$-category with pushouts and a terminal object.
Let $I$ be an object and $A$ a family of pointed objects over $I$.
Then the suspension of $\bigsqcup_i A(i)$ is a wedge of suspensions, naturally
in $A$:
\[
	\Sigma\bigsqcup_i A(i) \simeq \Sigma I \vee \Sigma\bigvee_i A(i).
\]
\end{lemma}
Phrased less idiosyncratically, this means that if $X \to Y \to Z$ is a cofibre sequence
where the map $X \to Y$ has a retraction, then $\Sigma Y \simeq \Sigma X \vee \Sigma Z$.
\begin{proof}
We refer to the following diagram in which every square is cocartesian.
The bottom right square gives the desired result.
\[\begin{tikzcd}[ampersand replacement=\&,cramped]
	I \& {\bigsqcup_i A_i} \& I \& 1 \\
	1 \& {\bigvee_i A_i} \& 1 \& {\Sigma I} \\
	\& 1 \& {\Sigma \bigvee_i A_i} \& {\Sigma \bigsqcup_i A_i}
	\arrow[from=1-1, to=1-2]
	\arrow[from=1-2, to=1-3]
	\arrow[from=1-3, to=1-4]
	\arrow[from=1-4, to=2-4]
	\arrow[from=2-4, to=3-4]
	\arrow[from=2-3, to=3-3]
	\arrow[from=1-3, to=2-3]
	\arrow[from=1-2, to=2-2]
	\arrow[from=1-1, to=2-1]
	\arrow[from=2-1, to=2-2]
	\arrow[from=2-2, to=3-2]
	\arrow[from=3-2, to=3-3]
	\arrow[from=3-3, to=3-4]
	\arrow[from=2-3, to=2-4]
	\arrow[from=2-2, to=2-3]
\end{tikzcd}\]
\end{proof}
\begin{lemma}\label{sigma-wedge-commute}
Let $\mathcal E$ be an $\infty$-category with pushouts and a terminal object.
Let $I$ be an object and $A$ a family of pointed objects over $I$.
Then we have
\[
	\Sigma \bigvee_i A(i) \simeq \bigvee_i \Sigma A(i).
\]
\end{lemma}
Here $\Sigma A(i)$ refers to the suspension of $A$ in the slice $\mathcal E/I$.
\begin{proof}
Consider the following cube, induced by functoriality of pushouts.
\[\begin{tikzcd}[ampersand replacement=\&,cramped]
	{\bigsqcup_i A(i)} \&\&\& I \\
	\&\& I \&\&\& {\bigsqcup_i \Sigma A(i)} \\
	\\
	{\bigvee_i A(i)} \&\&\& 1 \\
	\&\& 1 \&\&\& {\Sigma \bigvee_i A(i)}
	\arrow[from=5-3, to=5-6]
	\arrow[from=4-1, to=5-3]
	\arrow[from=4-4, to=5-6]
	\arrow[from=1-4, to=2-6]
	\arrow[from=1-1, to=2-3]
	\arrow[from=1-1, to=4-1]
	\arrow[from=1-1, to=1-4]
	\arrow[from=2-6, to=5-6]
	\arrow[from=1-4, to=4-4]
	\arrow[from=4-1, to=4-4]
	\arrow[from=2-3, to=5-3, crossing over]
	\arrow[from=2-3, to=2-6, crossing over]
\end{tikzcd}\]
The top, bottom, left, and back faces are cocartesian. By pasting, so is the
front face. This gives the required result.
\end{proof}
\begin{lemma}\label{wedge-distributivity}
Let $\mathcal E$ be an $\infty$-category with products and pushouts. 
Let $A$ and $B$ be two pointed objects over $I$.
Then we have
\[
	\bigvee_i (A(i) \vee B(i)) \simeq \bigvee_i A(i) \vee \bigvee_i B(i).
\]
\end{lemma}
\begin{proof}
The wedge $\bigvee_i$ from $(\mathcal E/I)_\star$ to $\mathcal E_\star$ has
a right adjoint, given by $A \mapsto A \times I$. So it preserves all colimits.
In particular it preserves binary coproducts, i.e.\ binary wedge.
\end{proof}

We are now ready to prove the main goal, that the map \eqref{disj-fw} splits
after suspension. By \cref{suspension-disj}, including the naturality statement,
its suspension is
\[
	\Sigma U_n \vee \Sigma \bigvee_{(i_1\ldots i_n) \in U_n} \fw(\Omega A(i_1)\ldots\Omega A(i_n)) \to 
				\Sigma U_n \vee \Sigma \bigvee_{(i_1\ldots i_n) \in U_n} (\Omega(A(i_1) \times \cdots \times \Omega(A(i_n))))
\]
By pushout pasting, we can ignore the first component, $\Sigma U_n$.
We can also commute $\Sigma$ and $\bigvee_i$ by \cref{sigma-wedge-commute}.
By a relative version of \cref{pushout-product-splits}, in the slice $\mathcal E/U_n$, 
   the resulting map looks like
   \[
   \bigvee_{(i_1\ldots i_n) \in U_n} \Sigma \fw(\Omega A(i_1) \ldots \Omega A(i_n)) \to
   \bigvee_{(i_1\ldots i_n) \in U_n} (\Sigma \fw(\Omega A(i_1) \ldots \Omega A(i_n)) \vee 
			   \Sigma \Omega A(i_1) \smsh \cdots \smsh \Omega A(i_n) )
  \]
Finally, by \cref{wedge-distributivity}, this is a pushout of
	\[ 1 \to 
		\bigvee_{(i_1\ldots i_n) \in U_n} \Sigma \Omega A(i_1) \smsh \cdots \smsh \Omega A(i_n).
	\]
Since $\Sigma$ preserves pushout squares, we also have that the suspension of $Q_{2n-2} \to Q_{2n}$
splits.
Thus we arrive at the following splitting for $\Sigma \Omega \bigvee_i A(i)$.

\begin{theorem}\label{loop-wedge-splitting}
Let $\mathcal E$ be an $\infty$-category with finite limits, colimits, and sequential colimits where
pushouts are universal and satisfy descent, and sequential colimits commute with pullbacks.
Let $I$ be a 0-truncated object of $\mathcal E$ with complemented diagonal,
and let $A$ be a pointed object over $I$.
Then we have the following stable splitting for the loop space of the wedge of $A$:
\[
	\Sigma \Omega \bigvee_{i \in I} A(i)
	\simeq \Sigma \bigvee_{n \in \N} \bigvee_{(i_1\ldots i_n) \in U_n} \Omega A(i_1) \smsh \cdots \smsh \Omega A(i_n).
\]
\end{theorem}
Recall that $U_n$ denotes the subobject of $I^n$ consisting of lists with no two adjacent elements equal.
\begin{proof}
Consider the zigzag construction as discussed above.
We have $\Omega \bigvee_{i \in I} A(i) \simeq \colim_n Q_{2n}$ and since suspension commutes with sequential colimits,
   $\Sigma \Omega \bigvee_{i \in I} A(i) \simeq \colim_n \Sigma Q_{2n}$
   The statement follows from the fact that each map in this sequence, $\Sigma Q_{2n-2} \to \Sigma Q_{2n}$, splits
   with the specified cofibre.
\end{proof}
\cref{loop-wedge-splitting} can be compared with the classical fact that, if
$G$ is an $I$-indexed family of groups, then each element of the free product
$\bigast_i G_i$ has a unique representation given by a number $n \in \N$, a
sequence of indices $(i_1,\cdots,i_n) \in U_n$, and a sequence $a_1\ldots a_n$
with $a_j \in G_{i_j}$ non-trivial.

\section{Outlook}\label{conclusion}
The zigzag construction suggests several directions for future research.
On the one hand one might look for direct applications. We have presented some
applications, but expect there to be many others.
It is natural to ask when we get stable splitting in the zigzag construction,
and relatedly one might look for a description of (co)homology.
The fact that the zigzag construction describes path spaces of pushouts
as sequential colimits of \emph{pushouts} suggests that it could be iterated,
to obtain results about iterated loop spaces of pushouts.
This would perhaps be especially interesting in homotopy type theory,
where not many qualitative results about higher identity types of pushouts are known.

On the other hand one might look for generalisations. Classically, in spaces,
models are known for $\Omega^n \Sigma^n X$~\cite{milgram,gils}. It remains
to be understood if such models can also be described in homotopy type
theory, or something like the elementary setting of this paper. Relatedly,
one might wonder if it is possible to explain higher-dimensional versions of
the Blakers--Massey theorem using a generalisation of the zigzag
construction. Yet another direction for generalisation would be to look for
descriptions of pushouts of $\infty$-categories as opposed to
$\infty$-groupoids.

Any argument presented in homotopy type theory suggests the possibility of
formalisation in a proof assistant. 
A version of the zigzag construction has been formalised and
proven correct by Vojtěch Štěpančík in \texttt{agda-unimath}~\cite{agda-unimath},
but most of the results in this paper remain unformalised at the moment.
\printbibliography
\end{document}